\documentclass[a4paper]{article}

\usepackage{amsmath}
\usepackage[symbol*]{footmisc}
\usepackage{amssymb}
\usepackage{amsthm}
\usepackage{pstricks}
\usepackage{mathdots}
\usepackage{epsfig}
\usepackage[left=2.5cm,top=3cm,right=3cm,bottom=2cm]{geometry}

\DefineFNsymbolsTM{myfnsymbols}{
  \textasteriskcentered *
  \textdagger    \dagger
  \textdaggerdbl \ddagger
  \textsection   \mathsection
  \textbardbl    \|%
  \textparagraph \mathparagraph
}%
\setfnsymbol{myfnsymbols}

\newtheorem{proposition}{Proposition}[section]
\newtheorem{theorem}[proposition]{Theorem}
\newtheorem{lemma}[proposition]{Lemma}
\newtheorem{corollary}[proposition]{Corollary}

\newtheorem{problem}[proposition]{Problem}

\newenvironment{proofmain}{\textsc{Proof of Theorem~\ref{main}.}}{\hfill$\square$\medskip}

\newenvironment{proofliczbaszesc}{\textsc{Proof of Corollary~\ref{six}.}}{\hfill$\square$\medskip}

\numberwithin{equation}{section}

\DeclareMathOperator{\Gl }{Gl}

\DeclareMathOperator{\kk}{K}
\DeclareMathOperator{\emm}{M}

\title{Algebras with finitely many conjugacy classes of left ideals versus algebras of finite
 representation type}
\date{}

\author{Arkadiusz M\c{e}cel \and Jan Okni\'nski}

\begin{document}

\maketitle

\begin{center}
\end{center}

\begin{center}
\end{center}

\vspace{1cm}

\begin{abstract}
Let $A$ be a finite dimensional algebra over an algebraically
closed field with the radical nilpotent of index $2$. It is shown
that $A$ has finitely many conjugacy classes of left ideals if and
only if $A$ is of finite representation type provided that all
simple $A$-modules have dimension at least $6$.
\end{abstract}

\section{Introduction} Let $A$ be a finite dimensional unital
algebra over a field $\kk$ and let $U(A)$ denote the group of units
of $A$. Following \cite{okr}, we denote by $C(A)$ the semigroup of
conjugacy classes of left ideals of $A$, with respect to the
natural operation $[L_1][L_2] = [L_1L_2]$ for $L_{1},L_{2}\in
L(A)$. Here $L(A)$ stands for the set of left ideals of $A$ and
$[L]$ is the conjugacy class of $L\in L(A)$ in $A$. A study of
$C(A)$ is in part motivated by a general program of searching for
semigroup invariants of associative algebras \cite{mec}. The semigroup
 $C(A)$ is also related to the subspace semigroup of an associative
 algebra, studied in \cite{ok2, okp}, which is an analogue
of the semigroup of closed subsets in an algebraic monoid. In the
context of ring theory, various related actions of $U(A)$ have
been considered on a ring $A$, see \cite{hry}, and also
\cite{ha1,ha2,hir}, leading to certain finiteness conditions for
$A$.

Finite dimensional algebras seem to be of a particular interest
from the point of view of finiteness of $C(A)$. The class of
algebras with $C(A)$ finite includes in particular every algebra
of finite representation type, see \cite{okr}, Theorem~6. Recall
that these are algebras with finitely many isomorphism classes of
finite dimensional indecomposable left modules. On the other hand, 
the celebrated second Brauer-Thrall conjecture, proved
by Nazarova and Roiter, can be used to show that if $\kk$ is algebraically closed 
then the fact that $C(\emm_{n}(A))$ is
finite for every $n\geq 1$ implies that $A$ is of finite
representation type, \cite{okr}, Theorem~7. 

One of the motivations is that certain numerical invariants and certain structural invariants (in terms of the finite semigroup $C(A)$) might provide new tools in the study of considered classes of algebras. The second motivation, and the aim of this paper, is to find a new internal characterization of algebras of finite representation type. We do this for a natural class of algebras, namely, the
class of algebras $A$ over an algebraically closed field $\kk$ with
the radical $J(A)$ nilpotent of index $2$. A study of this class
of algebras is motivated on one hand by the role it plays in
representation theory of arbitrary algebras, see \cite{pie}, and
on the other hand by the fact that within this class the semigroup
$C(A)$ determines the algebra $A$ up to isomorphism, see
\cite{mec}, Theorem~1.2.
In particular, our work is motivated by the following problem.\\

\noindent {\bf Problem.} Let $A$ be a finite dimensional algebra
over an algebraically closed field. Assume that $J(A)^2=0$ and the
lattice $I(A)$ of ideals of $A$ is distributive (we then simply
say that $A$ is distributive). Find necessary and sufficient
conditions for $A$ in order that $C(A)$ is a finite semigroup.\\

Some partial results in this direction can be found in \cite{mec2}.
The approach adopted there is based on matrix problems arising from certain abstract combinatorial structures,
 called skeletons. However, the results of the present paper are based on a reformulation of the finiteness problem of $C(A)$ in the language of quivers and their representations.

Assume that $\kk$ is an algebraically closed field and $A$ is a
finite dimensional $\kk$-algebra. Fix a maximal subset $\{e_1, e_2,
\ldots , e_k\}$ of a complete set of orthogonal primitive
idempotents of $A$ such that $Ae_i \not\simeq Ae_j$ as left
$A$-modules, for any $i \neq j$. Then the
directed graph $\Gamma(A) = (\Gamma(A)_0,\Gamma(A)_1)$ is called the (ordinary) quiver of $A$ if
the set of vertices $\Gamma(A)_0 = \{1, 2, \ldots , k\}$ and given $a, b \in \Gamma(A)_0$,
the arrows $\alpha: a \to b$ that constitute the set $\Gamma(A)_1$ are in a bijective correspondence with the vectors in a basis of the $\kk$-vector space $e_a(J(A)/J(A)^2)e_b$.  With $\Gamma(A)$ one associates  the separated quiver $\Gamma^s(A) =
(\Gamma^s(A)_0,\Gamma^s(A)_1)$, where $\Gamma^s(A)_0 = \Gamma(A)_0 \times \{0, 1\}$ and $\Gamma^s(A)_1 = \{((i, 0),
(j, 1)) | (i, j) \in \Gamma^s(A)_1\}$.

Let us recall some known facts about $C(A)$. The first one is an obvious finiteness condition
connected with the lattice of two sided ideals $I(A)$ of $A$. It is known that $I(A)$ is distributive if and only if $I(A)$
is finite, see \cite{pie}, \S2.2, Exercise 4; and \S2.6, Exercise
3; whence this is a necessary condition for $C(A)$ to be finite.
In case $A$ is a basic algebra (that is, $A/J(A)$ is a direct
product of copies of the field $\kk$), a solution to the above
problem was given in \cite{okr}, Theorem 12.

\begin{theorem}[\cite{okr}, Theorem 12]\label{renner} Let $A$ be a finite dimensional
basic algebra over an algebraically closed field $\kk$. Assume that
$J(A)^2 = 0$ and the lattice of ideals of $A$ is distributive.
Then $C(A)$ is finite if and only if the separated quiver
$\Gamma^s(A)$ has no cycles (as an unoriented graph) and $\dim
(eJ(A))\leq 3$ for every primitive idempotent $e$ of $A$.
\end{theorem}

Clearly, the matrix algebra $A = \emm_n(\kk[x]/(x^2))$ is of
finite representation type for every $n\geq 1$, while it does not
satisfy the restriction on the dimension of $eJ(A)$ if $n>3$. It
is also easy to construct examples of algebras $A$ of infinite
type such that $C(A)$ is finite, see \cite{mec}, Example~4.7. This
can be accomplished via the following classical result.\\

\noindent \textbf{Theorem (Gabriel, \cite{pie}, 11.8).} \textit{Let $A$ be a finite dimensional algebra
over an algebraically closed field. Assume that $A$ is
distributive and  $J(A)^2 = 0$. Then $A$ is of finite
representation type if and only if the separated quiver
$\Gamma^s(A)$ is a disjoint union of Dynkin graphs of types
$\mathbb{A}_n,\mathbb{D}_n,\mathbb{E}_6,\mathbb{E}_7,\mathbb{E}_8$.}\\

In contrast to
Theorem~\ref{renner}, if the algebra is not basic, not only the
structure of $\Gamma^s(A)$ but also the sizes of the simple blocks
of $A/J(A)$ will play a key role.  In this context, we refer to \cite{mec2}, Theorem 2.8 for a far reaching extension of Theorem 1.1. Our main result reads as
follows.

\begin{theorem} \label{main} Let $A$ be a finite dimensional algebra
over an algebraically closed field $\kk$. Assume that $A/J(A)\cong
\emm_{r_1}(\kk) \oplus \emm_{r_2}(\kk) \oplus \cdots \oplus \emm_{r_k}(\kk)$,
with $r_{i}\geq 6$ for every $i$, and $J(A)^{2}=0$. Then $C(A)$ is
finite if and only if $A$ is of finite representation type.
\end{theorem}

The following consequence complements Theorem~7 in \cite{okr}, mentioned above.

\begin{corollary}\label{six}
Let $A$ be a finite dimensional algebra over an algebraically
closed field and let\break $J(A)^2 = 0$. Then $A$ is of finite
representation type if and only if the semigroup $C(\emm_6(A))$ is
finite.
\end{corollary}

In Section~\ref{skeleton} we discuss the structure of nilpotent left ideals
of a radical square zero finite dimensional algebra $A$ and prove that the set of conjugacy
classes of these ideals is in 1-1 correspondence with a certain set of isomorphism classes
of representations of a quiver dual to the separated quiver of $A$. The proof of
the main theorem, relying also
on a classical geometric argument of Tits is then derived in
Section~\ref{proof}. In Section 4 we discuss the possible generalizations of the main result, including the role played
by the number $6$, and we state some problems.

\normalcolor

\section{Conjugacy classes of left ideals} \label{skeleton}

It is well known that algebras of finite representation type are distributive, \cite{pie}, Theorem 6.7. So, in view of the comment before Theorem 1.1, in this section we assume that $A$ is a finite dimensional distributive algebra over
 an algebraically closed field $\kk$, and that $J(A)^2=0$.
 First, we establish some notation. Assume that
\begin{equation} A/J(A) \simeq \emm_{r_1}(\kk) \oplus
 \emm_{r_2}(\kk) \oplus \cdots \oplus \emm_{r_k}(\kk),
\label{szerokosc}\end{equation} where $r_i, k$ are positive
integers. According to the classical Wedderburn-Malcev theorem we
have the following decomposition of $A$ into the direct sum of
linear subspaces:
 $$A = A' \oplus J(A) = A_1 \oplus A_2
\oplus A_3 \oplus \ldots \ldots \oplus A_k \oplus J(A),$$ where $A_i \simeq \emm_{r_i}(\kk)$ and $A' \simeq A/J(A)$.
Let $J_{ij} = f_iJ(A)f_j$ where $f_i$ is the unit in $A_i$.

It is well known and easy to prove that since $A$ is distributive and $J(A)^2 = 0$, it follows that $J_{ij}$ are minimal
subbimodules of the $A_i-A_j$-bimodule $f_iAf_j$, provided they
are nonzero, see \cite{pie}, \S 11.8.  In this case $J_{ij}$ are also two-sided ideals of
$A$. They can be, therefore, identified with the linear spaces of
rectangular matrices $\emm_{r_i \times r_j}(\kk)$. In the matrix
notation the following decomposition follows.
\begin{equation} A  = A' \oplus J(A) = \begin{bmatrix} A_1 & 0  & \ldots &0 \\
0 & A_2 &  \ldots & 0 \\
\vdots & \vdots & \ddots & \vdots \\ 0 & 0 &  \ldots & A_k
\end{bmatrix}
\oplus \begin{bmatrix} J_{11} & J_{12}  & \ldots & J_{1k} \\
J_{21} & J_{22}  & \ldots & J_{2k} \\
\vdots & \vdots &  \ddots & \vdots \\
J_{k1} & J_{k2}  & \ldots & J_{kk}
\end{bmatrix}\label{postacx}\end{equation} Hence, for $n := r_1 +
r_2 + \ldots + r_k$, the algebra $A$ can be identified with a
subalgebra of  $\emm_n(\kk[x]/(x^2))$, where $x^2 = 0$, and
$J_{ij} = f_iJ(A)f_j$, provided that it is nonzero, can be
identified with the respective sets of rectangular matrices of
sizes $r_i \times r_j$ with entries in the two-sided ideal of the
ring $\kk[x]/(x^2)$, generated by the coset of $x \in \kk[x]$.

Let $e_1, e_2, \ldots, e_n$ be the consecutive diagonal
idempotents of rank one in $A$. Put
 $r_i = n_{i+1} - n_{i}$, where  $0 = n_1 < n_2 < \ldots < n_{k+1} = n$. Then the sets
 \begin{equation} E_i = \{e_{n_i + 1}, e_{n_i + 2} , \ldots, e_{n_{i+1}}\}, \label{zbiorye}\end{equation}
defined for $i = 1, 2, \ldots, k$, satisfy the condition: $Ae_pA = Ae_qA$
if and only if $e_p, e_q$ belong to the same $E_i$. Moreover, $f_i := e_{n_{i} + 1} + \ldots + e_{n_{i+1}}$
 is a two-sided unity element of the algebra $A_i$, for $i =
1, 2, \ldots, k$, and the elements of $E_i$ are diagonal idempotents of rank one in $A_i$.

Consider the decomposition $\kk^{r_1} \oplus \ldots
\oplus \, \kk^{r_k}$ of the linear space $\kk^n$. For $1 \leq i \leq k$ let $J_i
:= \{j_{1}, j_2, \ldots, j_{t_i}\}$ be the sets of indices such that $J_{ij} \neq 0 \Leftrightarrow
j \in J_i$. We also define $\mathbb{V}_i := \{v \in \kk^n \, | \,
\pi_j(v) = 0, j \notin J_i\},$ where $\pi_j: \kk^n \to \kk^{r_j}$
is the natural projection.
For every $i = 1, \ldots, k$ define \begin{equation} a_i := r_{j_1} + \ldots + r_{j_{t_i}}, \text{ if } J_i \neq \emptyset, \text{ and otherwise } a_i = 0.\label{ai}\end{equation} Then $\mathbb{V}_i$ is isomorphic to $\kk^{a_i}$ as a linear
space. Moreover, for $e \in E_i$, we have $eJ(A) \simeq \mathbb{V}_i$.

Indeed, if $e_p, e_q$ are both orthogonal idempotents of rank one, contained in the same $E_i$, and if
 $e_{ts}$ are matrix units in $A_i$, for $n_i + 1 \leq t, s \leq n_{i+1}$, then
 $e_p = e_{pp}$ and $e_q = e_{qq}$, where $e_{pq}e_{qq}J(A) = e_{pp}J(A)$ and $e_{qp}e_{pp}J(A)
= e_{qq}J(A)$. Hence $eJ(A) \simeq \mathbb{V}_i$, for all $e \in E_i$.\\

We proceed with the main step towards an interpretation of our problem in the language of representations of quivers. Let $\mathcal{M}_A$ denote the set of block matrices $M = (M_{ij}) \in M_{(a_1 + \ldots + a_k) \times n}(\kk)$ such that $M_{ij} \in M_{a_i \times r_j}(\kk)$ and $M_{ij} = 0$ if $J_{ij} = 0$, for $i,j=1,\ldots ,k$. On this set, we have natural actions (by left and right multiplication, respectively) of the groups $G_{l}= \prod_{i=1}^{k}{\Gl_{a_i}(\kk)}$ and $G_{r} = \prod_{i =1}^{k}{\Gl_{r_i}(\kk)}$. 

\begin{proposition}\label{reprezentacja} There exists a bijection between
 the set $N(A)$ of nilpotent left ideals of $A$ and the set of $k$-tuples of linear subspaces $(V_1, \ldots, V_k)$, where $V_i \subseteq \mathbb{V}_i$, for
$1 \leq i \leq k$.

Moreover, there is a bijection between the set $C_J(A)$ of conjugacy classes of nilpotent left ideals in $A$ and the set of $G_{l}-G_{r}$ orbits on the set $\mathcal{M}_A$ of block matrices $M = (M_{ij}) \in M_{(a_1 + \ldots + a_k) \times n}(\kk)$ such that $M_{ij} \in M_{a_i \times r_j}(\kk)$ and $M_{ij} = 0$ if $J_{ij} = 0$, for $i,j=1,\ldots ,k$.
\end{proposition}

\begin{proof}

Let $L \in N(A)$. Then, according to \eqref{postacx} we have
$$L = AL = (A' + J(A))L = A'L =  A_1L \oplus A_2L \oplus \ldots \oplus A_kL.$$ Since $J(A)^2 = 0$
and $A_iA_j = 0$, for $1 \leq i,j \leq k$, it is clear that: $AA_iL = A'A_iL = A_iA_iL = A_iL.$
 Thus, every left ideal $L$ of $A$ that is contained in $J(A)$ can be expressed as
a direct sum of left ideals $L_{i}=A_{i}L$ of $A$ that are contained in $A_iJ(A)$, respectively, for
 $1 \leq i \leq k$. Therefore, it is enough to prove that there exists a 1-1 correspondence between
left ideals of $A$ contained in $A_iJ(A)$ and linear subspaces of $\mathbb{V}_i$.

Moreover, 
$$L_{i} = f_{i}L= f_iL_{i} = \bigoplus\limits_{e \in E_i}eL_{i},$$ where $E_i$ is a set of primitive orthogonal idempotents
of rank one in $A_i$, according to our earlier notation. Also, if $e_p, e_q$ are idempotents
contained in the same $E_i$, then the subspaces $e_pL_{i}, e_qL_{i}$, treated as linear subspaces of $\kk^nx$,
are equal to $W_i x$, for some subspace
$W_i \subseteq \mathbb{V}_i$.

Conversely, for every $i$ choose a linear subspace $V_{i} \subseteq \mathbb{V}_i$. Consider the set $L_{i}$ of matrices of sizes
 $r_i \times n$, whose every row is (as a vector) contained in $V_{i}x$. It is clear, that this set forms a left ideal
of $A$, contained in $A_iJ(A)$. Thus $L=\oplus_{i=1}^{k}L_{i}$ is a left ideal of $A$. It is clear that this leads to a 1-1
correspondence, as stated in the first part of the proposition.

In order to prove the second statement, recall that $r_{1}+\cdots +r_{k}=n$ and consider a block matrix $M= (M_{ij}) \in M_{(a_{1}+\cdots +a_{k})\times n}(\kk)$ such that 
$M_{ij} \in M_{a_i \times r_j}(\kk)$ and the block $M_{ij} = 0$ if $J_{ij} = 0$. Denote by $M_{i}$ the $a_{i}\times n$ - submatrix of $M$ such that $M_i$ can be treated as a block matrix with consecutive blocks
$M_{ij}$, $j=1,\ldots, k$. Let $V_{i}\subseteq \kk^{n}$ be the linear span of the set of all rows of $M_{i}$. By the definition of the integers $a_{i}$ it is clear that $V_{i}\subseteq \mathbb{V}_i\simeq \kk^{a_{i}}$. 
Let $L$ be the left ideal of $A$ corresponding to the $k$-tuple of spaces $(V_{1},\ldots, V_{k})$. Then $f(M)=L$ defines a map $f: \mathcal{M}_A\rightarrow N(A)$. 

On the other hand, for every $L\in N(A)$, which is determined as in the first part of the proof by a $k$-tuple of spaces $V_{i}\subseteq \mathbb{V}_i\subseteq \kk^{n}$, we know that $\dim (V_{i})\leq a_{i}$. Hence, there exists a matrix $M$ such that $f(L)=M$ (we form matrices $M_{i} \in M_{a_i \times n}(\kk)$, by choosing as rows an arbitrary set of $a_{i}$ vectors that span $V_{i}$, for every $i=1,\ldots, k$; then they define $M$ in a natural way). So $f$ is a surjective map. 

Next, assume that two matrices $M,M'\in M_{(a_{1}+\cdots +a_{k})\times n}(\kk)$ are given and $f(M)=f(M')$. The latter is equivalent to $V_{i}=V_{i}'$ for every $i$, where $V_{1},\ldots, V_{k}$ and $V_{1}',\ldots, V_{k}'$ are the row spaces of the corresponding matrices $M_{1},\ldots, M_{k}$ and $M_{1}',\ldots ,M_{k}'$, respectively. This means that the reduced row echelon forms of the matrices $M_{i}$ and $M_{i}'$ are equal, for every $i$. Or, equivalently, there exists $g_{i}\in \Gl_{a_{i}}(\kk)$ such that $g_{i}M_{i}=M_{i}'$. This means that there exists $g\in G_{l}$ such that $gM=M'$ ($g$ is block diagonal with consecutive blocks $g_{1},\ldots ,g_{k}$). It follows that the set of $G_{l}$-orbits on ${\mathcal M}_{A}$ (acting by left multiplication) is in a bijection with the elements of $N(A)$.
 
It is clear that two left ideals $L,L'\in N(A)$ are conjugate in $A$ if and only if $L'=Lu$ for some $u\in U(A)$. 
Using \eqref{postacx} we can write
$$ U(A) = \begin{bmatrix} \Gl_{r_1}(\kk) & \ldots & 0 \\ \vdots & \ddots &
 \vdots \\ 0 & \ldots & \Gl_{r_k}(\kk) \end{bmatrix} + J(A).$$ Hence, we may assume that $u\in G_{r}$. 
 
It is clear that linear spaces $V_{1}',\ldots ,V_{k}'$, ($V_{i}'\subseteq \kk^{n}$), corresponding to $Lu$ are of the form $V_{1}u,\ldots, V_{k}u$, where $V_{1},\ldots ,V_{k}$ are the subspaces corresponding to $L$. On the other hand, for any matrix $M\in {\mathcal M}_{A}$ and any $u\in G_{r}$, the row blocks $M_{1}',\ldots, M_{k}'$ that form the matrix $Mu$ satisfy $M_{i}'=M_{i}u$ for every $i$. Clearly, the row space of the matrix $M_{i}u$ is equal to $V_{i}u$. Therefore, $f(Mu) = Lu$. Since the set of $G_{l}$-orbits on ${\mathcal M}_{A}$ is in a bijection with the elements of $N(A)$ (via $f$), this easily implies that  $f$ induces a bijection between double cosets $G_{l}MG_{r}$, $M\in {\mathcal M}_{A}$,  and the conjugacy classes of left ideals of $A$ contained in $J(A)$.
\end{proof}

Next we will show, that the problem of finiteness for the semigroup
$C(A)$ can be reformulated in the representation theoretical language. First, let us recall certain notions.\\

Recall that in our setting, the separated quiver $\Gamma^s(A)$ of $A$ consists of $2k$ vertices $\{1, 2, \ldots, k\} \times \{0,1\},$ and there exists an arrow $(i, \epsilon') \to (j, \epsilon'')$ in $\Gamma^s(A)$ if $\epsilon' = 0, \epsilon'' = 1$ and $e'_iJ(A)e'_j \neq 0,$ where $e'_i, e'_j$ are any primitive idempotents taken from the sets $E_i, E_j$, defined in \eqref{zbiorye}. In this case there are precisely $d_{ij}$ arrows $(i,0) \to (j,1)$, where $d_{ij} := \dim_{\kk}{e'_iJ(A)e'_j}.$ We recall from \cite{pie}, Corollary 2.4c, that if the lattice $I(A)$ is distributive, then $\dim_{\kk}{eJ(A)f} \leq 1$, for any primitive idempotents $e, f \in A$. Therefore, the number of arrows between any two vertices in the quiver of an algebra is always equal either to $0$ or $1$.

Also, recall that for a quiver $Q =  (Q_0, Q_1)$ with the set of vertices $Q_0$ and the set of arrows $Q_1$ one may consider a $\kk$-linear representation, or more briefly, a representation $M = (M_a,\phi_a)_{a \in Q_0, \alpha \in Q_1}$ of $Q$, where $M_a$ is a $\kk$-linear space and $\phi_{\alpha}: M_a \to M_b$ is a linear map, for $\alpha = (a,b)$.

Assume that $Q_0$ has $n$ elements $\{1, \ldots, n\}$ and $d = (d_1, \ldots, d_n) \in \mathbb{N}^n$. By $rep_d(Q)$ we denote the set of representations of $Q$ with $M_i = \kk^{d_i}$, for all $i \in Q_0$. Consider the following linear space.
\begin{equation}\mathbb{A}_Q(d) = \prod\limits_{(i,j) \in Q_1}{\emm_{d_{j} \times d_{i}}(\kk)}.\label{repspace}\end{equation}
This object is known as the representation space corresponding to the dimension vector $d$. It admits a natural structure of an affine space. One can define an action of the affine algebraic group $G(d) = \prod_{i \in Q_0}{\Gl_{d_i}(\kk)}$ on $\mathbb{A}_Q(d)$ via the conjugation formula:
\begin{equation}(g \cdot x)_{\alpha} = g_{j} \cdot x_{\alpha} \cdot g_i^{-1},\label{klasyrep}\end{equation} where $g = (g_i) \in G(d)$, whereas $x_{\alpha}$ and $(g \cdot x)_{\alpha}$ are the matrices standing on the $\alpha$-th coordinates in $\mathbb{A}_Q(d)$ and $g \cdot \mathbb{A}_Q(d)$, respectively,  where $\alpha = (i,j) \in Q_1$. It is well known that two representations $M$ and $N$ in $rep_d(Q)$ are isomorphic if and only if the (natural) representatives of  $M$ and $N$ in $\mathbb{A}_Q(d)$ belong to the same $G(d)$-orbit, (see \cite{ass3}, XX.2).

We may now come back to the algebra $A$ and to the finiteness problem of the semigroup $C(A)$.

\begin{proposition}\label{repr} For an algebra $A$ with $J(A)^2 = 0$, consider an element $d \in \mathbb{N}^{2k}$ of the form:
\begin{equation} d = (a_1, \ldots, a_k, r_1, \ldots, r_k), \label{vector}\end{equation} where $a_i$ are as in Proposition \ref{reprezentacja}. Let $Q = \overline{\Gamma^s(A)}$  be the quiver obtained from the separated quiver $\Gamma^s(A)$ by inverting all its arrows and let $\mathbb{A}_Q(d)$ stand for the algebraic variety corresponding to the set $rep_d(\overline{\Gamma^s(A)})$ of representations of $Q$ with the dimension vector $d$. There is a 1-1 correspondence between the set of conjugacy classes of nilpotent left ideals of $A$ and the set of orbits of the action \eqref{klasyrep} of $G(d)$ on $\mathbb{A}_Q(d)$.
\end{proposition}

\begin{proof} 

According to Proposition~\ref{reprezentacja} the set $C_J(A)$ of conjugacy classes of nilpotent left ideals of $A$ is in 1-1 correspondence with the set of $G_l - G_r$-orbits on the set  $\mathcal{M}_A$ of block matrices. Clearly, $\mathcal{M}_A \simeq \mathbb{A}_Q(d)$. Since $G(d) = G_l \times G_r$, the element $g_1 \cdot M \cdot g_2$ is a conjugate of $M$ under the action \eqref{klasyrep} with $g = (g_1, g_2^{-1}) \in G(d)$. Thus, the result follows.
\end{proof}

The finiteness of the set $C_J(A)$ is equivalent to the finiteness of $C(A)$, see \cite{mec}. As a result, we obtain the desired characterization of algebras $A$ such that $C(A)$ is finite in the language of representation theory of $A$.

\begin{corollary}\label{collrep} Let $A$ be a radical square zero finite dimensional distributive algebra over an algebraically closed field. Let $d$ be the dimension vector defined in \eqref{vector}. The following conditions are equivalent:
\begin{itemize}
\item[(1)] the semigroup $C(A)$ is finite,
\item[(2)] the number of isomorphism classes of representations in the set $rep_{d}(\overline{\Gamma^s(A)})$ is finite.
\end{itemize}
\end{corollary}

\section{Proof of the main theorem} \label{proof}

In this section we prove the main result of this paper. Recall that by the Tits quadratic form $q_Q$ of a quiver $Q = (Q_0, Q_1)$ with $Q_0 = \{1, 2, \ldots, n\}$ we mean the following integral form (see \cite{ass1}, VII.4):
\begin{equation} q_Q(d) = \sum\limits_{i \in Q_0}{d_i^2} - \sum\limits_{(i,j) \in Q_1}{d_id_j}, \quad \text{ where } \quad d = (d_1, \ldots, d_n)\in \mathbb{Z}^n \label{titsform}\end{equation}

Recall the action \eqref{klasyrep}, for the dimension vector $d$.
Then the first summand of the sum \eqref{titsform} is clearly the
dimension of the group $G(d)$ (as an algebraic variety) and the
second one is the dimension of the algebraic variety
$\mathbb{A}_Q(d)$ associated to the representation space $rep_d(Q)$.
By a classical argument of Tits it is known that if
$q_Q(d)\leq 0$, for some nonzero $d$ with nonnegative
coordinates, then the number of isomorphism classes in $rep_d(Q)$
is infinite. Indeed, this follows from the fact, that the set $F =
\{(a \cdot id_{d_1}, \ldots, a \cdot id_{d_n} , a^{-1} \cdot
id_{d_1}, \ldots, a^{-1} \cdot id_{d_n}) \, | \, a \in \kk^*\}, $
where $\kk^*$ stands for the multiplicative group of $\kk$, acts
trivially on the variety $\mathbb{A}_Q(d)$. Thus if $q_Q(d)
\leq 0$, then the dimension of the group $G(d)/F$ acting on the
variety $\mathbb{A}_Q(d)$ is smaller than the dimension of the
variety itself. Thus the number of elements in $rep_d(Q)$ is
infinite, see \cite{pie}, 8.8. Vectors $d$ such that $q_Q(d)
= 0$ constitute the, so-called, radical of the integral form
$q_Q$.

Let $\leq$ be the natural coordinate-wise order on the set
$\mathbb{Z}^n$, namely $(p_1, \ldots, p_n) \leq (q_1, \ldots,
q_n)$ if and only if $p_i \leq q_i$, for all $1 \leq i \leq n$.
Recall that $x \in \mathbb{Z}^n$ is called positive if $x \neq 0$
and $x_i \geq 0$ for all $i$. We will simply say that $x > 0$. The
following lemma is crucial.

\begin{lemma}\label{podkontur} Consider an algebra $A$ of the form \eqref{postacx} and its reversed-arrow separated graph $Q = \overline{\Gamma^s(A)}$. Assume that the semigroup $C(A)$ is finite. Let
$d = (d_1, \ldots, d_{2k})$, where $d_i = a_i$, for $1 \leq i \leq k$ and $d_{k+j} = r_j$, for $1 \leq j \leq k$. Consider a
subquiver $Q' \subseteq Q$ with vertices $i_1, \ldots, i_l \subseteq \{1, 2,
\ldots ,2k\}$ and $e = (d_{i_1}, \ldots, d_{i_l})$. Then no positive
vector $e' \leq e$ is in the radical of the quadratic form $q_{Q'}$.
\end{lemma}

\begin{proof} Consider the full subquiver $Q''$ of $Q$ with the same vertex set as $Q'$. Clearly $q_{Q'}(v) \geq q_{Q''}(v)$, for any dimension vector $v$.

 Let  $d_e$ be the dimension vector that agrees with $d$ on the coordinates $i_1, \ldots, i_l$ and its remaining coordinates are equal to $0$. Suppose that a positive vector $e' \leq e$ appears in the radical of $q_{Q'}$. Then $q_{Q''}(e') \leq 0$, and thus for the (obviously defined from $e'$) positive vector $d_{e'} \leq d_e \leq d$, we have $q_{Q}(d_{e'}) =q_{Q''}(e') \leq 0$. Then from the argument of Tits we know
that there are infinitely many isomorphism classes of representations in $rep_{d_{e'}}(Q)$. Consider the
dimension vector $d- d_{e'}$. Its coordinates are all nonnegative and we can consider the representation
space $rep_{d-d_{e'}} (Q)$. Then the (obviously defined) direct sum $rep_{d_{e'}} (Q) \oplus rep_{d-d_{e'}} (Q)$ is naturally contained in $rep_{d}(Q)$, see [1], III.1. Therefore $rep_{d}(Q)$ has infinitely many isomorphism classes. Since $C(A)$ was assumed to be finite, this contradicts Corollary \ref{collrep}. The assertion follows. \end{proof}

\noindent The proof of our main result follows naturally.\\

\noindent \begin{proofmain} One implication is clear. If $\Gamma^{s}(A)$ is a disjoint union of Dynkin graphs, then by
 Gabriel's theorem $A$ is of finite representation type and from Theorem~6 in \cite{okr} we know that
$C(A)$ is finite.

Assume now that the semigroup $C(A)$ is finite. From Corollary~\ref{collrep} we know that the number of isomorphism classes of representations of $\overline{\Gamma^s(A)}$ for the dimension vector $d$ is finite. We will show that the separated graph $\Gamma^s(A)$ of $A$ is a disjoint union of Dynkin
graphs.

Without losing generality we may assume that $Q = \Gamma^s(A)$ is connected. Assume, to the contrary, that $Q$ is not Dynkin. Then it is known that $Q$, as an unoriented graph, must contain one of the following Euclidean graphs (see \cite{ass1}, Lemma VII.2.1):

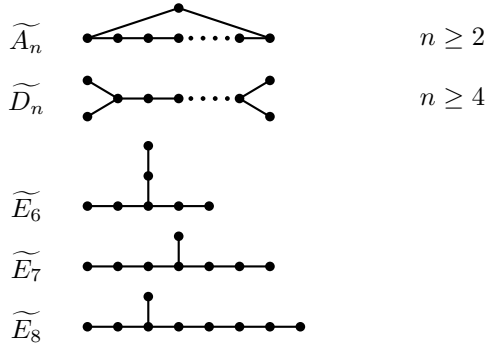
\begin{center}
\psset{unit=0.8cm}
\begin{pspicture}(0,4)(6.5,4.6)
\rput(0,4){\mbox{$\widetilde{A_n}$}} \rput(7,4){\mbox{$ n \geq 2$}}
\rput(0,3){\mbox{$\widetilde{D_n}$}} \rput(7,3){\mbox{$ n \geq 4$}}
\psdot[dotscale=1](1,4) \psline(1,4)(1.5,4) \psdot[dotscale=1](1.5,4) \psline(1.5,4)(2,4)
\psdot[dotscale=1](2,4) \psline(2,4)(2.5,4) \psdot[dotscale=1](2.5,4) \psdot[dotscale=0.5](2.7,4)
\psdot[dotscale=0.5](2.9, 4) \psdot[dotscale=0.5](3.1 ,4) \psdot[dotscale=0.5](3.3, 4)
\psdot[dotscale=1](3.5,4) \psline(3.5,4)(4,4) \psdot[dotscale=1](4,4)  \psdot[dotscale=1](2.5,4.5)
\psline(1,4)(2.5,4.5) \psline(4,4)(2.5,4.5)

\psdot[dotscale=1](1,3.3) \psdot[dotscale=1](1,2.7)  \psline(1,2.7)(1.5,3) \psline(1,3.3)(1.5,3)
\psdot[dotscale=1](1.5,3) \psline(1.5,3)(2,3) \psdot[dotscale=1](2,3) \psline(2,3)(2.5,3)
\psdot[dotscale=1](2.5,3) \psdot[dotscale=0.5](2.7,3) \psdot[dotscale=0.5](2.9, 3)
\psdot[dotscale=0.5](3.1 ,3) \psdot[dotscale=0.5](3.3, 3) \psdot[dotscale=1](3.5,3)
\psline(3.5,3)(4,3.3) \psline(3.5,3)(4,2.7) \psdot[dotscale=1](4,3.3) \psdot[dotscale=1](4,2.7)

\end{pspicture}
\end{center}

\begin{center}
\psset{unit=0.8cm}
\begin{pspicture}(0,-0.2)(6.5,4.3)
\rput(0,2){\mbox{$\widetilde{E_6}$}}
\rput(0,1){\mbox{$\widetilde{E_7}$}}
\rput(0,0){\mbox{$\widetilde{E_8}$}}

\psdot[dotscale=1](1,2) \psline(1,2)(1.5,2) \psdot[dotscale=1](1.5,2) \psline(1.5,2)(2,2)
\psdot[dotscale=1](2,2) \psline(2,2)(2,2.5) \psdot[dotscale=1](2,2.5) \psline(2,2.5)(2,3)
\psdot[dotscale=1](2,3) \psline(2,2)(2.5,2) \psdot[dotscale=1](2.5,2) \psline(2.5,2)(3,2)
\psdot[dotscale=1](3,2)

\psdot[dotscale=1](1,1) \psline(1.5,1)(1,1) \psdot[dotscale=1](1.5,1) \psline(1.5,1)(2,1)
\psdot[dotscale=1](2,1) \psline(2,1)(2.5,1) \psdot[dotscale=1](2.5,1) \psline(2.5,1)(2.5,1.5)
\psdot[dotscale=1](2.5,1.5) \psline(2.5,1)(3,1) \psdot[dotscale=1](3,1) \psline(3,1)(3.5,1)
\psdot[dotscale=1](3.5,1) \psline(3.5,1)(4,1) \psdot[dotscale=1](4,1)

\psdot[dotscale=1](1,0) \psline(1,0)(1.5,0) \psdot[dotscale=1](1.5,0) \psline(1.5,0)(2,0)
\psdot[dotscale=1](2,0) \psline(2,0)(2,0.5) \psdot[dotscale=1](2,0.5) \psline(2,0)(2.5,0)
\psdot[dotscale=1](2.5,0) \psline(2.5,0)(3,0) \psdot[dotscale=1](3,0) \psline(3,0)(3.5,0)
\psdot[dotscale=1](3.5,0) \psline(3.5,0)(4,0) \psdot[dotscale=1](4,0) \psline(4.5,0)(4,0)
\psdot[dotscale=1](4.5,0)

\end{pspicture}
\end{center}

Let $E$ be such an unoriented Euclidean subgraph. Let $e$ be
the projection of the dimension vector $d$ to the coordinates that
correspond to the vertices of $E$, as in Lemma \ref{podkontur}.
Observe that $r_i > 0$ for $1 \leq i \leq k$. Suppose that one of
the first $k$ coordinates of $d$, namely one of the numbers $a_i$,
is equal to $0$, see \eqref{ai}. We know that the $i$-th
coordinate of $d$, where $i \leq k$, corresponds to the vertex
$(i,0)$ of $\overline{Q}$. Assume that $(i,0)$ is a vertex of $E$.
By the definition of the separated quiver $Q$ of $A$, only
arrows of the form $(m,0) \to (j,1)$ may appear in $Q$ and since $E$
is a connected subquiver, an arrow of the form $(i,0) \to (j,1)$ must appear. Thus the
$i$-th block row of $J(A)$ in \eqref{postacx} cannot be zero.
Therefore $a_i \neq 0$ and all coordinates of $e$ are positive. By
the assumption on $A$ we know that the numbers $r_i$, appearing in
the decomposition of $A/J(A)$, are not smaller than $6$. Thus all
of the coordinates of $e$ are not smaller than $6$.

It is known that the radicals of the quadratic forms of the
quivers with the underlying graphs being Euclidean are of form
$\mathbb{Z}q$, where $q$ is of one of the following types,
see \cite{ass1}, Lemma VII.4.2:

\begin{center}
\psset{unit=0.7cm}
\begin{pspicture}(0,1)(15.5,-1)
\rput(-0.5,0){\mbox{$1$}}
\rput(0,0){\mbox{$1$}}
\psdot[dotscale=0.4](0.3,-0.1)
\psdot[dotscale=0.4](0.6,-0.1)
\psdot[dotscale=0.4](0.9,-0.1)
\rput(0.6,0.5){\mbox{$1$}}
\rput(1.2,0){\mbox{$1$}}
\rput(1.7,0){\mbox{$1$}}
\rput(2,0){\mbox{$,$}}

\rput(2.8,-0.3){\mbox{$1$}}
\rput(2.8,0.3){\mbox{$1$}}
\rput(3,0){\mbox{$2$}}
\psdot[dotscale=0.4](3.3,-0.1)
\psdot[dotscale=0.4](3.6,-0.1)
\psdot[dotscale=0.4](3.9,-0.1)
\rput(4.2,0){\mbox{$2$}}
\rput(4.4,-0.3){\mbox{$1$}}
\rput(4.4,0.3){\mbox{$1$}}
\rput(4.7,0){\mbox{$,$}}

\rput(5.5,-0.5){\mbox{$1$}}
\rput(6,-0.5){\mbox{$2$}}
\rput(6.5,-0.5){\mbox{$3$}}
\rput(6.5,0){\mbox{$2$}}
\rput(6.5,0.5){\mbox{$1$}}
\rput(7,-0.5){\mbox{$2$}}
\rput(7.5,-0.5){\mbox{$1$}}
\rput(7.8,0){\mbox{$,$}}

\rput(8.5,-0.3){\mbox{$1$}}
\rput(9,-0.3){\mbox{$2$}}
\rput(9.5,-0.3){\mbox{$3$}}
\rput(10,-0.3){\mbox{$4$}}
\rput(10,0.3){\mbox{$2$}}
\rput(10.5,-0.3){\mbox{$3$}}
\rput(11,-0.3){\mbox{$2$}}
\rput(11.5,-0.3){\mbox{$1$}}
\rput(11.8,0){\mbox{$,$}}

\rput(12.5,-0.3){\mbox{$2$}}
\rput(13,-0.3){\mbox{$4$}}
\rput(13.5,-0.3){\mbox{$6$}}
\rput(13.5,0.3){\mbox{$3$}}
\rput(14,-0.3){\mbox{$5$}}
\rput(14.5,-0.3){\mbox{$4$}}
\rput(15,-0.3){\mbox{$3$}}
\rput(15.5,-0.3){\mbox{$2$}}
\rput(16,-0.3){\mbox{$1$}}
\rput(16.3,0){\mbox{$.$}}

\end{pspicture}
\end{center}
where each of the generators listed above is presented with accordance to the structure of the respective graph: $\widetilde{\mathbb{A}_n}$, $\widetilde{\mathbb{D}_n}$, $\widetilde{\mathbb{E}_6}$, $\widetilde{\mathbb{E}_7}$ and $\widetilde{\mathbb{E}_8}$.

Let $r$ be a generator of the radical for the Euclidean graph $E$ considered above. Since all coordinates of $e$ are greater than or equal to $6$, then $r \leq e$. From Lemma~\ref{podkontur} it follows that $C(A)$ is infinite and we arrive to the contradiction with the supposition that $\Gamma^s(A)$ was not Dynkin. This concludes the proof.
\end{proofmain}

Finally, we derive Corollary~\ref{six}, which extends Theorem~7 in
\cite{okr}.\\

\noindent \begin{proofliczbaszesc} If $A$ is of finite representation type then so is $\emm_6(A)$ and $C(\emm_6(A))$ is finite by
Theorem 7 in \cite{okr}. On the other hand, if $C(\emm_6(A))$ is
finite then the algebra  $\emm_6(A)$ is distributive by Theorem~6
in~\cite{okr}. Moreover, it is clear that $J(\emm_6(A))^2 =
\emm_6(J(A))^2 = 0$ and that $\emm_6(A)$ satisfies the hypotheses
of Theorem~\ref{main}. Therefore, $\emm_6(A)$ is of finite
representation type by Theorem~\ref{main}. It follows that also
$A$ is of finite representation type, as desired.
\end{proofliczbaszesc}

A natural question arises whether the assertion of
Theorem~\ref{main} can be extended to a wider class of algebras,
and in particular whether the hypothesis $J(A)^2=0$ can be
dropped. From the point of view of our proof, the main problem
here is to find a translation of the conjugacy problem for left
ideals of $A$ contained in $J(A)$ in the language of an
appropriate matrix problem, as it is done in other classical situations, see for instance \cite{sim}.

\section{Remarks and questions}

We conclude with two remarks giving some insight into possible generalizations of our main result. They also put more light on the presence of number $6$ in our main result (though, the role played by this hypothesis is already clear from the proof of Theorem~\ref{main}).

As said before, in general it is not clear how to relate the conjugacy classes of left ideals to certain matrix problems, as we did with the radical square zero algebras. Nevertheless, there is a way of constructing matrix problems for any finite dimensional algebra that is a direct generalization of the one we have considered. For the dimension vector $d$, one may consider the representation space $rep_d(Q,I)$ of a quiver $Q$ bound by relations corresponding to an ideal $I\subseteq  J(\kk Q)^2$, see \cite{ass1}, III. The generalization of the geometric context of representations of quivers to the bound quivers case is quite natural, see \cite{ass3}, XX.2.

The use of quadratic forms in the general context of determining the representation type is common since the work of Gabriel \cite{gab2}. In \cite{bon2}, Bongartz introduced a generalization of the quadratic form of a quiver, by defining the Tits quadratic form $\widehat{q}_A: \mathbb{Z}^n \to \mathbb{Z}$ of the basic algebra $A$ with an acyclic quiver. It was shown that if the algebra $A$ has finite representation type then $\widehat{q}_A$ is weakly positive, that is: $\widehat{q}_A(x) > 0$, for all $x > 0$. The reverse implication is false in general, but it remains valid for many important classes of algebras of small global dimension: tilted algebras, double tilted algebras, quasitilted algebras, coil enlargements of concealed algebras, generalized multicoil algebras and others (see \cite{ass3}, XX.2).

How is the number $6$ related with these generalizations? A classical theorem of Ovsienko \cite{ovs} states that if a weakly positive integral quadratic form $q(x)$ has a root, that is such $x$ that $q(x) = 1$, then all its coordinates $x_i$ satisfy $x_i \leq 6$. Bongartz proved, that the following conditions are equivalent for the so-called simply connected algebras e.g. the algebras with the trivial fundamental group (for details, see \cite{ass3}, XX.2): (i) $A$ is representation finite, (ii) the Tits form $\widehat{q}_A$ of $A$ is weakly positive, (iii) $A$ does not admit a convex subalgebra $C$ which is critical. The latter is such an algebra $A$ that every proper convex subalgebra of $A$ is of finite representation type (see \cite{bon3}, for details). There exists a classification of critical algebras by means of the so-called Bongartz-Happel-Vossieck list, see \cite{ass2}, XIV. Each algebra $A$ on that list is defined as a path algebra of a bound quiver and a radical vector of the Tits quadratic form of $A$ is included. The coordinates of each of these radical vectors are less than or equal to $6$.
This suggests, that perhaps for the class of simply connected algebras the list of all possible ,,forbidden positive dimension vectors'' can be obtained similarly to the radical square zero case. These arguments motivate us to conclude this paper with the following problem.

\begin{problem} Consider finite dimensional distributive algebras $A$ over an algebraically closed field such that all simple $A$-modules have dimension at least $6$. Determine natural classes of such algebras for which the following conditions are equivalent:
\begin{itemize}
\item[(1)] the algebra $A$ is of finite representation type,
\item[(2)] the semigroup $C(A)$ is finite.
\end{itemize}
\end{problem}

\begin{tabular}{lll}
Arkadiusz M\c{e}cel & \quad \quad \quad \quad \quad & Jan Okni\'nski \\
\texttt{a.mecel@mimuw.edu.pl} & & \texttt{okninski@mimuw.edu.pl} \\
 & & \\
Institute of Mathematics & & \\
University of Warsaw & & \\
Banacha 2 & & \\
02-097 Warsaw, Poland & &
\end{tabular}

\end{document}